\documentclass[psamsfonts,11pt]{amsart}

\usepackage{amsmath}
\usepackage{amsthm}
\usepackage{amssymb}
\usepackage{bussproofs}
\usepackage{txfonts}

\theoremstyle{definition}
\newtheorem{theorem}{Theorem}

\newtheorem{definition}[theorem]{Definition}
\newtheorem{corollary}[theorem]{Corollary}
\newtheorem{example}[theorem]{Example}

\title{The modal logic of Reverse Mathematics}
\author{Carl Mummert}

\date{\today}
\address{C. Mummert, A. Saadaoui: Marshall University\\ 1 John Marshall Drive\\Huntington,~WV~25755}
\email{mummertc@marshall.edu}
\author{Alaeddine Saadaoui}
\author{Sean Sovine}
\address{S. Sovine: U.S. Army Corps of Engineers}
\begin{document}

\begin{abstract}
The implication relationship between subsystems in Reverse Mathematics has an underlying logic, which can be used to deduce certain new Reverse Mathematics results from existing ones in a routine way.
We use techniques of modal logic to formalize the logic of Reverse Mathematics into a system that we name s-logic. We argue that s-logic captures precisely the ``logical'' content of the implication and nonimplication relations between subsystems in Reverse Mathematics. We present a sound, complete, decidable, and compact tableau-style deductive system for s-logic, and explore in detail two fragments that are particularly relevant to Reverse Mathematics practice and automated theorem proving of Reverse Mathematics results.
\end{abstract}

\maketitle

\section{Introduction}

Reverse Mathematics is a research area in mathematical logic focusing on relationships between subsystems of second-order arithmetic~\cite{Simpson1999}. Here a \emph{subsystem} is simply a consistent theory in the language $L_2$ of second order arithmetic. In a typical result, a researcher focuses on two subsystems $S$ and $T$, each of which is rich enough to include a standard base system of axioms. The goal of the research is to show that the subsystem $S$ \emph{implies} the subsystem $T$ (that is, every $L_2$-structure that satisfies $S$ also satisfies $T$) or that the subsystem $S$ does not imply a subsystem $T$ (there is an $L_2$-structure that satisfies $S$ but does not satisfy $T$).  As usual, if $S$ and $T$ are subsystems of second order arithmetic, we write $S \vdash T$ if every $L_2$-structure that satisfies $S$ also satisfies $T$, and $S \not\vdash T$ if there is an $L_2$-structure that satisfies $S$ but does not satisfy~$T$.  Because the completeness theorem for first-order logic applies to second-order arithmetic, it would be equivalent to write $S \vDash T$.

To study the $\vdash$ and $\not \vdash$ relations from a purely logical viewpoint, we will employ a formal \emph{strict implication} symbol $\strictif$ and its formal negation, $\not\strictif$.  We consider a logic, which we call \emph{s-logic}, whose formulas are of the forms $A \strictif B$ and $A \not\strictif B$, where $A$ and $B$ are formulas of propositional logic.  In an intended interpretation of a formula of s-logic, the propositional variables are assigned to subsystems of second-order arithmetic, $\strictif$ is interpreted as $\vdash$, and $\not\strictif$ is interpreted as $\not\vdash$. Our goal is to study the logic of such formulas, giving a sound and complete deductive system and establishing compactness and decidability theorems. 

There has been a significant amount of previous research on the strict implication operator, $\strictif$. This research was initiated by Lewis~\cite{Lewis1914,LL1959} and continued by many others including Barcan~\cite{Barcan1946} and Hacking~\cite{Hacking1963} before being subsumed into the general theory of modal logic. The most common contemporary approach, which we also follow, treats $A \strictif B$ as an abbreviation for the modal formula $\Box(A \to B)$.  We have not found previous research that treats precisely the fragment of modal logic necessary for Reverse Mathematics, however. We are interested in formulas of both forms $A \strictif B$ and $A \not \strictif B$, not only formulas for the first form, as some authors have been. But we are not interested in formulas with nested strict implications, such as $A \strictif (B \strictif C)$, as other authors have been. If we interpret $\strictif$ as $\vdash$ in a formula of that sort, the inner strict implication must be replaced by a formalized provability predicate, and we would arrive at a provability logic somewhat related to the one studied by Solovay~\cite{Solovay1976}. We are interested only in the logic of the actual provability relation, $\vdash$, and thus we wish to avoid formulas in which strict implications and nonimplications are nested.

Although our motivation for studying s-logic comes from Reverse Mathematics, s-logic may also be applied to other areas of mathematics. As a concrete example, one could identify propositional variables with properties that an arbitrary topological space may possess, interpret $S \strictif T$ to mean that every space with property $S$ has property $T$, and  interpret $S \not\strictif T$ to mean there is a space with property $S$ that does not have property $T$.  The logic corresponding to this topological interpretation of $\strictif$ and $\not\strictif$ will be the same as the logic for the Reverse Mathematics interpretation.  It is easy to think of additional interpretations for which the same logic is obtained. 

One intended application of our research is in automated theorem proving of Reverse Mathematics results.  While many Reverse Mathematics results require original arguments, there are other results implicit in the literature that are obtained by routine combination of results from several papers. Thus, as the volume of research in Reverse Mathematics continues to increase, it can be tedious to determine whether a particular question has been implicitly resolved. A website known as the Reverse Mathematics Zoo, maintained by Damir D. Dzhafarov, contains a list of many Reverse Mathematics results from the literature, and uses these to automatically deduce some of the additional Reverse Mathematics results implicit in the known ones.  We hope that a more complete understanding of the underlying logic will help the development of such systems. The results of the final section, in particular, deal with fragments of s-logic that are relevant to automated analysis of the Reverse Mathematics literature. 

The research presented here was initiated by the third author in an undergraduate research project and continued by the second author as a master's thesis. The first author supervised both of these projects. The first and second authors then extended the results to their present form.  

\section{Reverse mathematics, modal logic, and s-logic}\label{sec2}

In this section, we present and justify the syntax and semantics of s-logic, and establish a semantic compactness theorem.   The syntax begins with a choice of an alphabet of propositional variables. In our intended interpretations, each propositional variable will represent a subsystem of second-order arithmetic. 

\begin{definition}
A \emph{signature} for s-logic consists of a infinite (possibly uncountable) set $\Sigma$ of propositional variables along with the non-variable symbols `(', `)', `$\land$', `$\lor$',  `$\to$', `$\lnot$', `$\strictif$', and `$\not\strictif$'. 
\end{definition}

For the remainder of the paper, we will assume that some particular signature has been fixed.

\begin{definition}[s-formulas]
The \emph{propositional formulas} are the smallest set of formulas such that:
\begin{enumerate}
\item Each propositional variable is a propositional formula.
\item If $A$ is a propositional formula, so is $\lnot A$.
\item If $A$ and $B$ are propositional formulas, so are $(A \land B)$, $(A\lor B)$, and $(A \to B)$.
\end{enumerate}

An \emph{s-formula} is of the form $A\strictif B$ or $A \not\strictif B$, where $A$ and $B$ are propositional formulas.  A formula of the form $A \strictif B$ is a \emph{strict implication}, while a formula of the form $A \not\strictif B$ is a \emph{strict nonimplication}.  An \emph{s-theory} is an arbitrary set of s-formulas.
\end{definition}


To motivate our choice of semantics, consider an $L_2$-structure $M$. If each propositional variable is associated with a subsystem, we may form a valuation $w_M\colon \Sigma\to \{T,F\}$ by putting $w_M(X) = T$ if and only if $
M \vDash X$.  Of course, if $X \vdash Y$, then $M$ will satisfy $X \to Y$. But, if $X \not\vdash Y$, a particular $L_2$ structure $M$ might still satisfy $X \to Y$. In particular, all the subsystems normally considered in Reverse Mathematics are true in the standard model of second-order arithmetic. In general, to have valuations that witness the consistency of strict nonimplications, we will need to look at a semantics that uses sets of valuations, which we call frames. 

If $\mathcal{M}$ is a set of $L_2$-structures, we may form the associated frame $\{w_M : M \in \mathcal{M}\}$. Under the definitions we will give, this frame will satisfy an s-formula $A \strictif B$ if every structure in $\mathcal{M}$ satisfies $A \to B$, and will satisfy $A\not\strictif B$ if there is a structure in $\mathcal{M}$ that satisfies $A$ and does not satisfy $B$.  Frames of this kind, which are arise from sets of $L_2$-structures, are the intended interpretations of s-logic.   

Our goal, however, is to reason in a \textit{logical} manner about the relationships between subsystems, in a way that is compatible with our limited knowledge at each moment of time. At each moment, a researcher knows about a particular set of $L_2$-structures, but does not know about all $L_2$-structures. Moreover, for each $L_2$-structure $M$ that has been studied, the researcher knows the truth values within $M$ of particular subsystems, but does not know the truth values of all subsystems. For example, there are some subsystems whose consistency is an open problem. If $X$ is such a subsystem, the researcher must consider for the sake of logical analysis both valuations that make $X$ true and ones which make $X$ false, as long as these valuations are consistent with all other known results. This analysis leads to a very general semantics for s-logic, with a constructive character. 

\begin{definition}[Valuations and frames]
A \emph{valuation} is a function from the set of propositional variables to the set $\{T,F\}$. As usual, each valuation can be extended uniquely to a valuation that assigns a truth value to each propositional formula. 

A \emph{frame} is a nonempty set of valuations.  A strict implication $A \strictif B$ is \emph{satisfied} by a frame $R$ if, for every valuation $w \in R$, $w(A \to B) = T$. This is equivalent to: for every $w \in R$, either $w(A) = F$ or $w(B) = T$. A strict nonimplication $A \not\strictif B$ is satisfied by $R$ if there is at least one valuation $w \in R$ such that $w(A) = T$ and $w(B) = F$.  A frame satisfies an s-theory $\Gamma$ if every formula in $\Gamma$ is satisfied by the frame. 
\end{definition}

The semantics for s-logic uses all possible frames. Although the intended interpretation of $\strictif$ is $\vdash$, they differ in important ways when arbitrary frames are considered.  For example, if each propositional variable from a fixed alphabet is associated with a subsystem of second-order arithmetic, and a frame $R$ satisfies a given set $\Gamma$ of s-formulas on that alphabet, there may not be a set of $L_2$-structures $\mathcal{M}$ with $R = R_\mathcal{M}$, because there may be relationships between the subsystems that are not stated in $\Gamma$. For example, if $A$ and $B$ are subsystems such that $A \vdash B$, then every frame of the form $R_{\mathcal{M}}$ satisfies $A \strictif B$; but $\Gamma$ may not contain $A \strictif B$ and $R$ may not satisfy that formula.  Similarly, if $A$ and $B$ are subsystems such that $A \not \vdash B$, a frame of the form $R_\mathcal{M}$ will satisfy $A\not\strictif B$ if and only if there is an $L_2$-structure in $\mathcal{M}$ that satisfies $A$ and does not satisfy $B$.  

These differences are to be expected. If we translate several Reverse Mathematics results into a set of s-formulas, and then formally derive consequences from these formulas, we cannot expect to derive all possible Reverse Mathematics results, but only the ones that can be proven by looking at the logical structure of formulas, without considering the meanings of the propositional variables within them. In other words, we only expect to formally derive new formulas that are, in a sense, routine combinations of existing formulas.  Similarly, if we begin with only a fixed collection of $L_2$-structures, $\mathcal{M}$, we cannot expect to use formal methods of s-logic to derive the existence of a new $L_2$-structure. Thus we expect that, when we define a deductive system for s-logic, if an s-formula $A \not \strictif B$ can be derived from a set of s-formulas $\Gamma$, then among any collection of $L_2$ structures $\mathcal{M}$ for which $R_\mathcal{M}$ satisfies $\Gamma$, at least one of the structures in $\mathcal{M}$ must satisfy $A$ and not satisfy $B$.

\subsection{Relationship with modal logic}

Although the motivation for our semantics does not directly come from modal logic, our definition of a frame can be viewed as a slight modification of Kripke semantics in modal logic. Under our semantics, an s-formula $A \strictif B$ corresponds exactly to the modal $\Box(A \to B)$, where $\Box \phi$ holds in a frame if and only if $\phi$ holds in all valuations of the frame. However, because we are not interested in formulas with nested modal operators, we have no need for an accessibility relation in our definition, and we do not require the full forcing relation $\Vdash$. For readers accustomed to modal logic, our system can be viewed as analogous to a fragment of S5, in that a strict implication or strict nonimplication is ``visible'' from every world (valuation) in the frame.

We could thus employ a general deductive system for modal logic (such as S5) to study s-logic. There are several disadvantages to that approach, which lead us to reject it. The first is that we look for a deductive system whose intensional aspects match the intended interpretation more closely. A proof in S5 may require significant reinterpretation to be read as a result of reverse mathematics, but the deductive systems we will present match the intension of the intended interpretation, so that a proof in these systems is easily read as a proof in the usual style of Reverse Mathematics. The second disadvantage is that general modal logic includes formulas with nested modal operations, such as $\Box(A \to \Box(\lnot B))$. Such formulas have no place in the intended interpretation, because we seek to interpret $\strictif$ as the actual provability relation, not as a formalized provability relation.

\subsection{Compactness of s-logic}

In the next section we will establish a sound and complete deductive system for s-logic. As a preliminary result, we first establish a semantic compactness theorem which will be useful in our later proofs.

\begin{theorem}[Compactness] If every finite subset of an s-theory is satisfiable, then the entire s-theory is satisfiable.
\end{theorem}
\begin{proof}
The proof uses the so-called ``standard interpretation'' of modal logic into first-order logic~\cite{BB2007}. This interpretation converts each s-formula into a first-order formula in such a way that an s-theory is satisfiable if and only if the corresponding first-order theory is satisfiable. The compactness theorem for s-logic then follows immediately from the compactness theorem for first-order logic.
%
\end{proof}

The proof of the compactness theorem suggests that we could also form a deductive system for s-logic by interpreting s-logic into first-order logic. The deductive systems for first-order logic are even farther from the intended interpretation of s-logic, however.


\section{Tableau system}

Our first inference system is inspired by the system of Mints~\cite{Mints92}.
It is a refutational system in the unsigned tableau style.  One motivation for this type of derivational system is that the proof (refutation) technique closely matches the way that a researcher in Reverse Mathematics might analyze a routine combination of results. Moreover, it is known in the automated theorem proving community that software-generated tableaux can be effectively converted into natural-language prose proofs of their results. 

For convenience, we use a slightly different set of formulas to label the nodes of a tableau.  We first fix a \emph{world alphabet}, which is an infinite set of variables that can be used to symbolize worlds (valuations) in a hypothetical frame. 

\begin{definition} Let $W$ be a fixed world alphabet. The \emph{tableau formulas} consist of all strict implication and strict nonimplication formulas, and all expressions of the form $(A,w)$, where $A$ is a propositional formula and $w\in W$. 
\end{definition}

\begin{definition} A \emph{tableau} for a set $\Gamma$ of tableau formulas is a finite tree $T$, with each node labeled by a (possibly infinite) set of tableau formulas, such that the root of $T$ is labeled with $\Gamma$ and each non-root node is obtained from its parent by one of the tableau inference rules in Figure~\ref{fig1}.  Here, when the rule $\not\strictif$ is applied, $v$ must be an element of $W$ that is not mentioned in the ancestor nodes of the node where the rule is being applied. When the rule $\strictif$ is applied, $w$ may be any element of $W$. 

A branch (path) through a tableau is \emph{closed} if it contains a node for which the label contains both $(A,w)$ and $(\lnot A,w)$ for some propositional formula $A$ and some $w \in W$.  A tableau is closed if every maximal branch is closed. 
\end{definition}

Intuitively, the labels on each node of a tableau represent assertions about a possible frame. A strict implication is asserted to hold in all valuations of the frame; a strict nonimplication is asserted to hold in some, unspecified, valuation; and a tableau formula $(A,w)$ asserts that $A$ holds in valuation $w$. 

\begin{figure}
\begin{center}
\begin{tabular}{lcl}
     \AxiomC{$\Gamma, (A,w)$}
     \AxiomC{$\Gamma, (B,w)$}
     \LeftLabel{$\lor$\quad}
     \BinaryInfC{$\Gamma, (A\lor B,w)$}
\DisplayProof
&
\parbox{0.5in}{\hskip1em}
&
     \AxiomC{$\Gamma, (\lnot A,w), (\lnot B,w)$}
     \LeftLabel{$\lnot \lor$\quad}
     \UnaryInfC{$\Gamma, (\lnot(A\lor B),w)$}
\DisplayProof
\\[3em]
     \AxiomC{$\Gamma, (A,w), (B,w)$}
     \LeftLabel{$\land$\quad}
     \UnaryInfC{$\Gamma, (A\land B,w)$}
\DisplayProof
&
\parbox{0.5in}{\hskip1em}
&
     \AxiomC{$\Gamma, (\lnot A,w)$}
     \AxiomC{$\Gamma, (\lnot B,w)$}
     \LeftLabel{$\lnot \land$\quad}
     \BinaryInfC{$\Gamma, (\lnot(A\land B),w)$}
\DisplayProof
\\[3em]
     \AxiomC{$\Gamma, (\lnot A,w)$}
     \AxiomC{$\Gamma, (B,w)$}
     \LeftLabel{$\to$\quad}
     \BinaryInfC{$\Gamma, (A\to B,w)$}
\DisplayProof
&
\parbox{0.5in}{\hskip1em}
&
     \AxiomC{$\Gamma, (A,w), (\lnot B,w)$}
     \LeftLabel{$\lnot \to$\quad}
     \UnaryInfC{$\Gamma, (\lnot(A\to B),w)$}
\DisplayProof
\\[3em]
     \AxiomC{$\Gamma, (\lnot A,w)$}
     \AxiomC{$\Gamma, (B,w)$}
     \LeftLabel{$\strictif$\quad}
     \BinaryInfC{$\Gamma, A\strictif B$}
\DisplayProof
&
\parbox{0.5in}{\hskip1em}
&
     \AxiomC{$\Gamma, (A,v), (\lnot B,v)$}
     \LeftLabel{$\not\strictif$\quad}
     \RightLabel{\quad($v$ new)}
     \UnaryInfC{$\Gamma, A\not\strictif B$}
\DisplayProof
\\[3em]
       \AxiomC{$\Gamma, (A,w)$}
     \AxiomC{$\Gamma,(\lnot A,w)$}
     \LeftLabel{$C$\quad}
     \BinaryInfC{$\Gamma$}
\DisplayProof  
&
\parbox{0.5in}{\hskip1em}
& \AxiomC{$\Gamma, (A,w)$}
     \LeftLabel{$\lnot\lnot$\quad}
     \UnaryInfC{$\Gamma, (\lnot\lnot A,w)$}
\DisplayProof
\end{tabular}
\end{center}
\caption{Tableau-style inference rules}\label{fig1}
\end{figure}
\begin{example}\label{ex1}
The following diagram shows a closed tableau using the world alphabet $W = \{ w_1 \}$.
The root node, at the bottom, is labeled with $X \not\strictif Y, X \strictif A, B \strictif Y, A \strictif B$. Each inference is labeled with the corresponding rule from Figure~1. For convenience, formulas on a node are not re-written on the descendants of that node.  The symbol $\otimes$ indicates a closed branch.
\begin{center}
\begin{prooftree}
\renewcommand{\extraVskip}{4pt}
\AxiomC{$\otimes$}
\noLine
\UnaryInfC{$(\lnot X, w_1)$} 

\AxiomC{$\otimes$}
\noLine
\UnaryInfC{$(\lnot A, w_1)$} 

\AxiomC{$\otimes$}
\noLine
\UnaryInfC{$(B, w_1)$} 

\LeftLabel{$\strictif$}
\BinaryInfC{$(\lnot B, w_1)$}  

\AxiomC{$\otimes$}
\noLine
\UnaryInfC{$(Y, w_1)$} 

\LeftLabel{$\strictif$}
\BinaryInfC{$(A, w_1)$} 

\LeftLabel{$\strictif$}
\BinaryInfC{$(X, w_1), (\lnot Y, w_1)$} 

\LeftLabel{$\not\strictif$}
\UnaryInfC{$X \not\strictif Y, X \strictif A, B \strictif Y, A \strictif B$} 
\end{prooftree}
\end{center}
\vskip1em

The reason that only one symbol is needed in the world alphabet in this deduction is that there is only one nonimplication formula listed at the root of the tableau.
\end{example}

\begin{theorem}[Soundness]\label{thm:sound}
Suppose that there is a tableau for a set $\Gamma$ of s-formulas such that every branch of $\Gamma$ is closed. Then no frame can satisfy~$\Gamma$.
\end{theorem}
\begin{proof}
The proof is by induction on the structure of the tableau, with one case for each of the ten tableau rules. For each rule, it can be shown directly that if a frame $R$ satisfies the set of formulas on the bottom of the rule, then the frame also satisfies at least one of the sets of formulas on the top of the rule. Here, each time a new world variable $v$ is introduced at a particular node, $v$ is interpreted on that node and all of its descendants as a particular valuation $w_v$ in $R$, and $R$ satisfies $(A,v)$ if and only if $w_v(A) = T$. 
\end{proof}

%
%

The hypothesis of finiteness in they following theorem is a convenience that will be removed in Theorem~\ref{lemsc}. For applications to automated theorem proving, the finite case is of the most interest. 

\begin{theorem}[Completeness]\label{thm:complete}
Suppose that $\Gamma$ is a finite set of s-formulas such that there is no closed tableau for $\Gamma$. Then there is a frame that satisfies~$\Gamma$. 
\end{theorem}
\begin{proof}
Let $\Gamma$ be a finite set of s-formulas. We begin by forming a finite tableau $T$ such that, whenever a formula $A$ appears on a maximal branch, the corresponding tableau rule for $A$ is also applied on that branch, and such that for every propositional formula $A$ and world variable $w$ that appears on a maximal branch, the rule $C$ is applied to that branch using the formula $A$ and world variable~$w$. Such a tableau can be made by repeatedly applying tableau rules in a systematic way until the desired conditions are met, and the resulting tableau will be finite so long as rule $C$ is only applied to a formula $A$ and world variable $w$ that already appear on a branch. 

If there is no closed tableau for $\Gamma$, then in particular $T$ does not close, so there is at least one maximal branch $B$ in $T$ which is not closed.  Then, for every world variable $v$ that appears on $B$, we define a valuation $w_v$. 
For each propositional letter $X$ that appears on $B$, the terminal node of $B$ contains either $(X,v)$ or $(\lnot X,v)$, by construction. Because $B$ is not closed, only one of these cases can occur. We let $w_v(X) = T$ in the former case, and $w_v(X) = F$ in the latter.   Let $R$ be the frame that contains the valuations $w_v$ for all world variables $v$ that appear on~$B$. 

It can then be shown directly by induction from the terminal node of $B$ back to the root that $R$ satisfies the bottom set of formulas in each tableau rule that was used to form the branch $B$. Thus $R$ satisfies the set of tableau formulas at the root of $B$, so $R$ satisfies~$\Gamma$. 
\end{proof}

The notation from the next definition will be used to simplify the statements of several theorems. 

\begin{definition}
The \emph{strict negation} of a s-formula $\phi$, denoted $- \phi$, is defined by cases: $-(A \strictif B)$ is $A \not\strictif B$, and 
 $-(A \not\strictif B)$ is $A \strictif B$.\end{definition}

Unlike the negation symbol $\lnot$, which is part of the language of propositional logic, strict negation is strictly a notation in the metalanguage; the symbol `$-$' is never part of an s-formula. The key property is that a frame satisfies an s-formula $A$ if and only if the frame does not satisfy~$-A$.

We now turn to the issue of characterizing logical consequence in s-logic. 
\begin{definition}
An s-formula $A$ is a \textit{strict consequence} of an s-theory $\Gamma$ if every frame that satisfies $\Gamma$ satisfies $A$. 
\end{definition}

\begin{theorem}\label{lemsc}
An s-formula $A$ is a strict consequence of an s-theory $\Gamma$ if and only if there is a closed tableau for $\Gamma \cup \{-A\}$. 
\end{theorem}
\begin{proof}
If there is a closed tableau for $\Gamma \cup \{-A\}$ then, by the soundness theorem, there is no frame that satisfies $\Gamma \cup \{-A\}$, and thus every frame that satisfies $\Gamma$ satisfies $A$.

For the converse, suppose that every frame that satisfies $\Gamma$ satisfies~$A$. Then no frame satisfies $\Gamma \cup \{-A\}$. By the compactness theorem, this means that there is a finite subset $\Delta$ of $\Gamma \cup \{-A\}$ that is not satisfied by any frame.  By the completeness theorem, there is a closed tableau for $\Delta$. This tableau becomes also a closed tableau for $\Gamma \cup \{-A\}$ if label on the root of the tableau is changed from $\Delta$ to~$\Gamma \cup \{-A\}$, with similar changes to the remaining nodes.
\end{proof}

\begin{example}
In light of Lemma~\ref{lemsc}, the closed tableau in Example~\ref{ex1} shows that $A \not\strictif B$ is a strict consequence of  $\{X\not\strictif Y, X \strictif A, B \strictif Y\}$, and also shows that $X \strictif Y$ is a strict consequence of $\{X \strictif A, A \strictif B, B \strictif Y\}$. In contrast, neither $B \strictif C$ nor $B \not \strictif C$ is a strict consequence of $\{A \strictif B, A \strictif C\}$.
\end{example}

\begin{theorem}[Decidability of s-logic]
Let $V$ be the set of pairs $(\Gamma, A)$ where $\Gamma$ is a finite s-theory, $\phi$ is an s-formula, and $\phi$ is a strict consequence of $\Gamma$. Then, under a standard G\"odel numbering of formulas and finite sets of formulas, the set $V$ is computable.
\end{theorem}
\begin{proof}
Given $(\Gamma, \phi)$, we may effectively form a finite tableau $T$ for $\Gamma \cup \{-\phi\}$, as in the proof of Theorem~\ref{thm:complete}. If this tableau is closed, then $\phi$ is a strict consequence of $\Gamma$. If $T$ is not closed then, again as in the proof of Theorem~\ref{thm:complete}, $\phi$ is not a strict consequence of~$\Gamma$.
\end{proof}


\section{Two fragments}

In this section, we consider two fragments of s-logic that are of particular interest in the practice of Reverse Mathematics, and give short and natural deductive systems for these fragments.

\newcommand{\calF}{\mathcal{F}}

\begin{definition} Suppose that a set of propositional variables has been fixed. 
\begin{itemize}
\item $\calF_1$ consists of all s-formulas of the forms $X\strictif Y$ and $X \not\strictif Y$, where $X$ and $Y$ are individual propositional variables.
\item $\calF_2$ consists of all s-formulas of the forms $A \strictif Y$ and $A \not \strictif Y$, where $A$ is a nonempty conjunction of propositional variables and $Y$ is a single propositional variable. 
\end{itemize}
\end{definition}

Fragment $\calF_1$ corresponds, in a sense, to the pure implicational and nonimplicational part of s-logic, in which all propositional connectives have been removed. Fragment $\calF_2$ is motivated by results in Reverse Mathematics such as the theorem that $\mathsf{RT}^2_2$ is equivalent to $\mathsf{SRT}^2_2 + \mathsf{COH}$~\cite{CJS2001}.   It is known that $\mathsf{RT}^2_2$ implies both $\mathsf{COH}$ and $\mathsf{SRT}^2_2$, and their conjunction implies $\mathsf{RT}^2_2$, but neither $\mathsf{COH}$ nor $\mathsf{SRT}^2_2$ implies $\mathsf{RT}^2_2$. These facts can be expressed via the following s-theory in $\calF_2$:%
\newcommand{\yyskip}{\quad}%
\begin{align*}
\{ &
\mathsf{SRT}^2_2 \land \mathsf{COH} \strictif \mathsf{RT}^2_2, \yyskip
\mathsf{RT}^2_2 \strictif \mathsf{SRT}^2_2,  \yyskip
 \mathsf{RT}^2_2 \strictif \mathsf{COH},  \\ &
 \mathsf{SRT}^2_2 \not\strictif \mathsf{RT}^2_2,   \yyskip
\mathsf{COH} \not\strictif \mathsf{RT}^2_2 \}.
\end{align*}
Surveying the Reverse Mathematics literature shows that almost all published results on implications or nonimplications between subsystems can be translated into s-theories in $\calF_2$. It is thus worthwhile to consider abbreviated sets of inference rules that are sound and complete for $\calF_1$ and $\calF_2$.

We will state sound and complete deductive systems for these fragments. Such systems are particularly useful in automated theorem proving for enumerating the consequences of a given s-theory. We begin with $\calF_2$.  For notational convenience, if $A$ and $B$ are conjunctions of variables, we may write $A \land B$ for the conjunction obtained by inserting $\land$ between $A$ and $B$. 

\begin{definition}\label{rules2}
The deductive system for $\calF_2$ consists of four inference rules (I), (W), (HS), and (N).
 Intuitively, rule (W) allows for weakening of hypotheses and rule (HS) is a version of the hypothetical syllogism. 
\smallskip
\begin{center}
\begin{tabular}{rp{4in}}
I: & For any propositional variable $X$, deduce $X \strictif X$. \\
W: & From $A \strictif Y$, deduce $B\strictif Y$, where $B$ is any conjunction such that every conjunct of $A$ is also a conjunct of $B$. \\
HS: & 
From $X \land B \strictif Y$ and $A \strictif X$, deduce $A \land B \strictif Y$.\\
N: & 
From $A \not \strictif X$, $A \land Z \strictif X$, and $A \strictif Y$ for each conjunct $Y$ of $B$, deduce $B \not \strictif Z$. 
\end{tabular}
\end{center}
Each of these rules is a scheme: $A$ and $B$ may be replaced by arbitrary conjunctions of propositional variables, while $X$, $Y$, and $Z$ may be replaced by arbitrary propositional variables. In rule (HS), the conjunction $B$ may be empty.
\end{definition}

It is straightforward to verify that the rules are sound: if a frame satisfies $\Gamma$, and $\phi$ is derivable from $\Gamma$ with the rules, then the frame satisfies $\phi$. We next verify that these rules give a complete deductive system for $\calF_2$.

\begin{theorem}[Completeness for $\calF_2$]\label{completef2} Suppose that $\Gamma$ is a consistent set of s-formulas in $\calF_2$, $\phi$ is an s-formula in $\calF_2$, and every frame that satisfies $\Gamma$ satisfies~$\phi$. Then there is a derivation of $\phi$ from $\Gamma$ using the rules in Definition~\ref{rules2}.
\end{theorem}

\begin{proof}  Working towards a contradiction, we assume there is no derivation of $\phi$ from $\Gamma$ with the stated rules. Because the rules are sound, we may thus assume that $\Gamma$ is closed under the rules and $\phi \not\in\Gamma$. The proof has two cases, depending on whether $\phi$ is a strict implication or a strict nonimplication.

\textbf{Case 1:} $\phi$ is of the form $C \strictif Z$, where $C$ is a nonempty conjunction. It suffices to construct a valuation $w_C$ that satisfies $\Gamma$ and does not satisfy $\phi$. To this end, we define a valuation
\[
w_C(X) = 
\begin{cases}
T & \text{if } C \strictif X \in \Gamma, \\
F & \text{if } C \strictif X \not \in \Gamma.
\end{cases}
\]
We must verify that $w_C$ satisfies every strict implication $U_1 \land \cdots \land U_k \strictif V$ in $\Gamma$. To do so, suppose that $w_C(U_i) = T$ for all $i \leq k$. Then, for each $i \leq k$, we have that $C \strictif U_i \in \Gamma$. Now, by applying rules (HS) and (W) repeatedly, we may derive $C \strictif V$. For example, we may first derive $C \land U_2 \land \cdots \land U_k \strictif V$ via rule (HS), then derive 
\[
U_2 \land C \land U_3 \land \cdots \land U_k \strictif V
\]
via rule (W), then derive 
 \[
 C \land C \land U_3 \land \cdots \land U_k \strictif V
  \]
via rule (HS), and continue in this way until at the end we derive $C \strictif V$ by rule~(W). Thus $w_C$ satisfies every strict implication in~$\Gamma$.  

It remains to verify that $w_C$ does not satisfy $\phi$. For each conjunct $Y$ of $C$, we may derive $C \strictif Y$ by rules (I) and (W), and thus $w_C(Y) = T$.  However, because $\phi \not \in \Gamma$, we have $w_C(Z) = F$. Thus $w_C$ is a valuation that satisfies $\Gamma$ but does not satisfy $\phi \equiv C \strictif Z$. This completes the first case, because we may adjoin $w_C$ to any frame satisfying $\Gamma$ to yield a larger frame that does not satisfy $\Gamma \cup \{\phi\}$. 

\textbf{Case 2:} $\phi$ is of the form $C \not\strictif Z$, where $C$ is a nonempty conjunction. We will build a frame that satisfies $\Gamma$ and which contains no valuation satisfying $C \not \strictif Z$.  It is sufficient to show that for each strict nonimplication $D \not\strictif Y$ in $\Gamma$ there is a valuation $w$ satisfying the strict implications of $\Gamma$ in which $w(D) = T$, $w(Y) = F$, and either $w(C) = F$ or $w(Z) = T$.  We may then take one such valuation for each strict nonimplication in $\Gamma$ to construct a frame satisfying $\Gamma$ but not $\phi$. 

We thus fix a strict nonimplication $D \not \strictif Y$ in $\Gamma$. If there is any valuation satisfying $\Gamma$ in which 
$w(D) = T$, $w(Y) = F$, and $w(C) = F$, we are done. Therefore, we may safely assume that, for each conjunct $U$ of $C$, every valuation $w$ that satisfies $\Gamma$ and has $w(D) = T$ and $w(Y) = F$ will have $w(U) = T$.  We claim that, under this assumption, we have that $D \strictif U$ is in $\Gamma$. To see this, consider
the valuation $w_D$ defined in the same way as $w_C$ from Case~1. We have
that $w_D$ satisfies every strict implication in $\Gamma$ and, for each variable $X$, $w_D(X) = T$ if and only if $D \strictif X$ is in~$\Gamma$. Because $D \not\strictif Y$ is in $\Gamma$, and $\Gamma$ is consistent, $D \strictif Y$ is not in $\Gamma$, so $w_D(Y) = F$.  Thus, under our most assumption, $w_D(U)$ must be true, which means that $D \strictif U$ is in $\Gamma$.

Now, consider the valuation $w_{D \land Z}$. We have $w_{D\land Z}(Z) = T$ and $w_{D \land Z}(D)  = T$. It follows from the previous paragraph that $w_{D \land Z}(C) = T$ as well. There are two subcases. Subcase 1: $w_{D \land Z}(Y) = F$. In this
case, $w_{D \land Z}$ satisfies $D \not\strictif Y$ but does not satisfy $C \not\strictif Z$ 
(because $Z$ is true) and we are done. Subcase 2: $w_{D\land Z}(Y) = T$. In this subcase, we have that $D \land Z \strictif Y$ is in $\Gamma$. Because we also have 
$D \not\strictif Z \in \Gamma$ and $D \strictif U \in \Gamma$ for every conjunct $U$ 
of $C$, we may apply rule~(N) to show that $C \not\strictif Z$ is in $\Gamma$, which is a contradiction.
\end{proof}

\begin{corollary} If $\Gamma$ is a consistent s-theory in $\calF_2$ and $\phi$ is in $\calF_2$ then $\phi$ is a strict consequence of $\Gamma$ if and only if $\phi$ can be derived from $\Gamma$ using the rules of Definition~\ref{rules2}.
\end{corollary}

We now turn to fragment $\calF_1$. The inference rules for this fragment are simplified versions of the rules for $\calF_2$. Because hypotheses of s-formulas in $\calF_1$ are simply propositional variables, the weakening rule (W) is no longer necessary. 

\begin{definition}\label{rules1}
The deductive system for $\calF_1$ consists of the following three rules (I), (HS), and (N):
\smallskip
\begin{center}
\begin{tabular}{rl}
I: & For any propositional variable $X$, deduce $X \strictif X$. \\
HS: &  From $X  \strictif Y$ and $Y \strictif Z$, deduce $X \strictif Z$.
\\
N: & From 
$X \not\strictif Y$, $X \strictif W$, and $Z \strictif Y$, deduce $W \not\strictif Z$.
\end{tabular}
\end{center}
In each of these rules, $W$, $X$, $Y$, and $Z$ may be replaced with arbitrary propositional variables.
\end{definition}

It is straightforward to verify that these rules are sound. The completness proof is parallel to the one for~$\calF_2$.

\begin{theorem}[Completeness for $\calF_1$] Suppose that $\Gamma$ is a consistent set of s-formulas in $\calF_1$, $\phi$ is an s-formula in $\calF_1$, and every frame that satisfies $\Gamma$ satisfies~$\phi$. Then there is a derivation of $\phi$ from $\Gamma$ using the rules in Definition~\ref{rules1}.
\end{theorem}

\begin{proof}
The proof is parallel to the proof of Theorem~\ref{completef2}. 
As before, we assume that $\Gamma$ is closed under the deduction rules and $\phi \not\in\Gamma$. The proof again divides into two cases. The first case, when $\phi$ is a strict implication, is extremely similar to the first case of Theorem~\ref{completef2}.

For the second case, it is sufficient to show that whenever $W \not\strictif Z \not\in\Gamma$ and $X \not\strictif Y \in \Gamma$, there is a valuation satisfying all strict implications in $\Gamma$, and satisfying $X \not\strictif Y$, in which $W$ is false or $Z$ is true. We may assume without loss of generality that every valuation that satisfies the strict implications in $\Gamma$ and also satisfies $X\not\strictif Y$ must satisfy~$W$.  Then, defining the valuation $w_X$ as in Case~1 of Theorem~\ref{completef2}, we see that $w_X(W) = T$, and thus $X\strictif W$ is in $\Gamma$.

Now consider the following valuation:
\[
w_{X,Z}(U) = 
\begin{cases}
T & \text{if $X \strictif U \in \Gamma$ or $Z \strictif U \in \Gamma$,}\\
F & \text{otherwise.}
\end{cases}
\]
We first verify that $w_{X,Z}$ satisfies each strict implication $U \strictif V$ in $\Gamma$.
If $w_{X,Z}(U) = T$, then either $X \strictif U \in \Gamma$ or 
$Z \strictif U \in \Gamma$. Then, because $\Gamma$ is closed under rule (HS), we have $X \strictif V$ or $Z \strictif V$ is in~$\Gamma$, respectively. 
Thus $w_{X,Z}(U \to V) = T$, as desired. Hence $w_{X,Z}$ satisfies all strict implications in~$\Gamma$.

Now we have $w_{X,Z}(X) = T$, $w_{X,Z}(Z) = T$, and $w_{X,Z}(W) = T$ because $X \strictif W \in \Gamma$.  If $w_{X,Z}(Y) = F$ then we are done. We show that this must happen by assuming that $w_{X,Z}(Y) = T$. 
Then either $X \strictif Y \in \Gamma$ or $Z \strictif Y \in \Gamma$. The former is impossible because $X \not\strictif Y \in \Gamma$ and $\Gamma$ is consistent. Thus $Z \strictif Y \in \Gamma$. But we also have $X \not \strictif Y \in\Gamma$ and $X \strictif W \in \Gamma$, so we may derive $W \not \strictif Z \in \Gamma$ by rule (N). This is a contradiction. Subcase 2: $w_{X,Z}(Y) = F$. Then 
$w_{X,Z}$ is the desired valuation. 
\end{proof}

\begin{corollary} If $\Gamma$ is a consistent s-theory in $\calF_1$ and $\phi$ is in $\calF_1$ then $\phi$ is a strict consequence of $\Gamma$ if and only if $\phi$ can be derived from $\Gamma$ using the rules of Definition~\ref{rules1}.
\end{corollary}

\bibliographystyle{amsplain}
\providecommand{\bysame}{\leavevmode\hbox to3em{\hrulefill}\thinspace}
\providecommand{\MR}{\relax\ifhmode\unskip\space\fi MR }
\providecommand{\MRhref}[2]{%
  \href{http://www.ams.org/mathscinet-getitem?mr=#1}{#2}
}
\providecommand{\href}[2]{#2}

\end{document}